\documentclass[preprint,12pt,superscriptaddress]{elsarticle}

\usepackage{amsmath,amssymb,amsthm,mathtools}
 \usepackage{enumitem} 
 \usepackage{microtype} 
 \usepackage[hidelinks]{hyperref}

  \emergencystretch=2em \hypersetup{ pdftitle={Square matrices with integer eigenvalues for all permutations of integer entries}, pdfauthor={Shuming Cheng, Siran Zhang}, pdfkeywords={integer matrix, integral spectrum, entry permutations, $2$-adic valuation} }

\makeatletter \apptocmd{\ps@pprintTitle}{\g@addto@macro\@oddfoot{\unskip}}{}{} \makeatother

\newtheorem{theorem}{Theorem}[section]
\newtheorem{proposition}[theorem]{Proposition} 
\newtheorem{lemma}[theorem]{Lemma}
 \newtheorem{corollary}[theorem]{Corollary}
  \theoremstyle{definition} 
   \newtheorem{definition}[theorem]{Definition}
    \newtheorem{conjecture}[theorem]{Conjecture}
     \theoremstyle{remark} 
     
      \theoremstyle{problem} 
     \newtheorem{problem}[theorem]{Problem}

\newcommand{\Q}{\mathbb Q} 
 
\newcommand{\Z}{\mathbb Z}
 \newcommand{\F}{\mathbb F}

  \newcommand{\PSI}{\operatorname{PSI}}
  \newcommand{\one}{\mathbf 1}

\begin{document}

\begin{frontmatter}

\title{Square matrices with integer eigenvalues under entry permutations}
	\author[1,2]{Shuming Cheng\corref{cor1}}
	\ead{drshuming.cheng@gmail.com}
	\author[1,3]{Siran Zhang}
	
	\cortext[cor1]{Corresponding author.}
	
		\affiliation[1]{organization={State Key Laboratory of Autonomous Intelligent Unmanned Systems, Shanghai Research Institute for Intelligent Autonomous Systems, Tongji University},
		city={Shanghai}, postcode={201203}, country={China}}
		\affiliation[2]{organization={Department of Control Science and Engineering, Tongji University},
		city={Shanghai}, postcode={201804}, country={China}}
	\affiliation[3]{organization={Research \& Development Center BMW, BMW China Services Ltd.},
		city={Beijing}, postcode={101318}, country={China}}

\begin{abstract} 
	
We investigate the problem of which multisets of integer entries yield integer eigenvalues in every square matrix arrangement? Combining module reduction with a dyadic criterion for complete splitting of cubic polynomials, we first show that if the multiset has at least one zero entry, the only possibilities have either at most one nonzero entry or exactly two whose product is a perfect square. We then apply block embedding to extend it to higher dimensions, by obtaining a linear threshold on the number of zeros. For any multiset close to a nonzero constant, we use rank reduction to derive an exact criterion for one exceptional entry, thus yielding nonconstant examples in infinitely many dimensions, and to also derive an exact reduction for two distinct exceptional entries in dimension three. The general fully nonzero three-dimensional case remains open.

 \end{abstract}


\end{frontmatter}

\section{Introduction}\label{sec:introduction}

Whether an integer matrix has integral spectrum, i.e., all its eigenvalues are integers, has been studied from several points of view. Such possibilities have been explored for symmetric matrices~\cite{Estes1992,CaoKoyuncu2016}, graph-related families~\cite{BarikBehera2024}, invariants-based constructions~\cite{Rushanan1995}, and general ones~\cite{Renaud1983,Cronin1987,TowseCampbell2016,KenyonEtAl2024}, while a complementary perspective is provided that almost all integer matrices have no integer eigenvalues~\cite{MartinWong2009} and matrices in dimension two can be prescribed with integer eigenvalues~\cite{MartinWong2008}. In most of these works, the spectral property depends on relations among particular matrix positions, however, an arbitrary permutation of matrix entries generally destroys such relations so that an integer matrix can lose its integer eigenvalues under entry permutation.

In this work, we study the problem of whether an integer square matrix and ones under all possibly entry permutations can simultaneously have integral spectrum. Requiring integral spectrum for every entry permutation of a matrix is essentially imposing a condition on the multiset formed by matrix entries, and correspondingly, our problem can be interpreted as what the structure of a matrix can survive this freedom to move each entry independently? 

Hall's work gives an intriguing starting point~\cite{Hall2025}. In dimension two, the Pythagorean triples are utilised to generate nonconstant multisets for which every square matrix arrangement has integer eigenvalues, and the four entries $1, 6, 7, 12,$ provide a simple example. Hall also derives a necessary and sufficient criterion involving six quadratic discriminants and asks whether other solutions exist in dimension three, beyond the constant multisets. The present work aims to answer this question.

First, we define a multiset of $n\times n$ integers as permutation-spectrally integral ($\PSI$) in dimension $n$, if every square matrix whose entries are exactly elements of this multiset has integral spectrum. Thus, our problem can be precisely reformulated as whether a nonconstant $\PSI$ multiset exists in $n\geq 3$. And if possible, what structures the $\PSI$ multiset should satisfy. 

Then, we combine module reduction with a dyadic criterion for complete splitting of cubic polynomials to show that the multiset density, defined as the number of nonzero entries in a multiset, plays a critical role in determining whether any given multiset is $\PSI$. In particular, if the multiset in dimension three has at least one zero, then it is $\PSI$ if and only if it has either at most one nonzero entry or exactly two whose product is an integer square. For the fully nonzero case, constant multisets are conjectured to be the only $\PSI$ examples.

We further apply block embedding to extend the above result to higher dimensions. When the nonzero entry proportion of a multiset in dimension $n\geq 3$ is less than $1-3/n$,  it is $\PSI$ if and only if it has at most two nonzero entries whose product is an integer square. Finally, we use rank reduction for the multiset close to a nonzero constant to show that the dense case behaves differently, that a $\PSI$ multiset with one exceptional entry exists in infinitely many dimensions, and a reduction criterion for a $\PSI$ multiset with two distinct exceptional entries in dimension three are also obtained.

The rest of this work is organized as follows. Section~\ref{sec:prelim} introduces some basic notations, recalls Hall's discriminant criterion in dimension two, and generalises some results to higher dimensions.  Sections~\ref{sec:three} presents the main result about the $\PSI$ multiset in dimension three, and Section~\ref{sec:embedding} extends it to the multiset in any dimension beyond three. Then, Section~\ref{sec:defects} presents nontrivial $\PSI$ constructions in higher dimensions, and finally, Section~\ref{sec:summary} concludes with a brief summary.

\section{Preliminaries}\label{sec:prelim}

Consider an $n\times n$ integer matrix and its $n^2$ entries form an integer multiset $\mathcal{S}$, and denote by $\mathcal M_n(\mathcal{S})$ the set of all $n\times n$ integer matrices whose entries are exactly the elements of $\mathcal{S}$. In $\mathcal M_n(S)$, any two distinct matrices can be transformed into each other under a series of entry permutations. A matrix has integral spectrum, if all its eigenvalues are integers. In this work, we study the problem of whether all matrices in $\mathcal M_n(\mathcal{S})$ can simultaneously have integral spectrum, which is close related to the multiset $\mathcal{S}$.

\begin{definition}
	An integer multiset $\mathcal{S}$ is permutation-spectrally integral in dimension $n$ ($\PSI_n$), if every matrix in $\mathcal M_n(\mathcal{S})$ has integral spectrum.
\end{definition}

Thus, our problem can be formulated as follows.

\begin{problem}\label{Problem}
	Is there an $\PSI$ multiset $\mathcal{S}$ for $n\geq 2$ ? And if there is, what conditions should the $\PSI_n$ multiset $\mathcal{S}$ satisfy?
\end{problem}

A useful property of the integer multiset is first noted in the following definition. Then, the known results for $n=2$ obtained by Hall~\cite{Hall2025} are recapped, and generalisations to an arbitrary dimension $n$ are also presented. 

\begin{definition}
	The density $\nu(\mathcal{S})$ of an integer multiset $\mathcal{S}$ is defined as the number of nonzero elements in $\mathcal{S}$, or equivalently, the number of nonzero entries of every matrix in $\mathcal{M}_n(\mathcal{S})$. Thus, $0\leq \nu(\mathcal{S})\le n^2$.
\end{definition}

 A square in $\Z$ means a nonnegative perfect square; a square in $\Q$ means the square of a rational number.  It follows from the rational root theorem~\cite{AhmadiEtAl2009,CostelloWilliams2016} that a rational root of a monic polynomial is an integer, and consequently, complete splitting of an integer characteristic polynomial over $\Q$ is equivalent to integral roots.

\subsection{The discriminant criterion in dimension two}

Any multiset $\mathcal{S}$ of $2\times 2$ integers gives rise to the square matrix in a general form of
\begin{equation}
	M=\begin{pmatrix} a & b\\ c & d \end{pmatrix}
	\end{equation}
with $a, b, c, d \in \Z$. The characteristic polynomial of $M$ is given by
\begin{equation}
\chi_M(\lambda)=\lambda^2-(a+d)\lambda+(ad-bc),
\end{equation} 
and the discriminant is
\begin{equation}
\Delta(M)=(a+d)^2-4(ad-bc)=(a-d)^2+4bc.
\end{equation}

By using algebra analysis, Hall showed that $M$ has two integer eigenvalues if and only if $\Delta(M)$ is a square in $\mathbb Z$, and thus derived the discrimination criterion for any $\PSI_2$ multiset in~\cite{Hall2025}.

\begin{proposition}[Hall's discriminant criterion]\label{prop:hall}	An integer multiset $\mathcal{S}=\{a,b,c,d\}$ is $\PSI$ in dimension $two$ if and only if the following six	integers are squares in $\Z$:
	\begin{equation}\label{eq:six}
		\begin{aligned}
			&(a-d)^2+4bc, &&(b-c)^2+4ad,\\
			&(a-b)^2+4cd, &&(c-d)^2+4ab,\\
			&(a-c)^2+4bd, &&(b-d)^2+4ac.
		\end{aligned}
	\end{equation}
\end{proposition}

It is interesting to note that Hall has also used the Pythagoras theorem to obtain a sufficient condition in the form of $a+b=c+d$ and its permutations, together with~(\ref{eq:six}), for the $\PSI_2$ multiset. However, it is not necessary, if the density $\nu(\mathcal{S})$ is less than $2$. Indeed, it follows directly from above that

\begin{corollary}\label{c2}
	For an integer multiset $\mathcal{S}$ with $\nu(\mathcal{S})\le2$, it is $\PSI_2$ if and only if it is either $\{a, 0, 0, 0\}$, i.e., $\nu(S)\le1$, or
	\begin{equation}\label{two}
		\{a, b, 0,0\},~~~~~\,{\rm where}~~\,a,b\ne0,~~ab=q^2,~~q\in\Z.
	\end{equation}
\end{corollary}

It is easy to verify that $a+b\neq 0$ for $a, b\neq 0$ or $a\neq b$, not satisfying the sufficient condition obtained in~\cite{Hall2025}.

\subsection{Generalisation to an arbitrary dimension}

Denote by $x^{[m]}$ for $m$ copies of $x$. Using combinatorial matrix theory~\cite{BrualdiCvetkovic2009} generalises Corollary~\ref{c2} to every dimension $n\geq 2$. 

\begin{proposition}\label{prop:two} For an integer multiset $\mathcal{S}$ with $n\ge2$ and $\nu(\mathcal{S})\le2$, it is $\PSI_n$ if and only if either $\mathcal{S}=\{a, 0^{[n^2-1]}\}$, i.e., $\nu(\mathcal{S})\le1$, or
	\begin{equation}\label{eq:two}
		\mathcal{S}=\{a, b, 0^{[n^2-2]}\},~~~~~\,{\rm where}~\,a,b\ne0,~ ab=q^2\in\Z.
	\end{equation}
\end{proposition}
\begin{proof}
	Consider first the $\nu(\mathcal{S})=1$ case where there is only one single nonzero element in $\mathcal{S}$ of $n^2$ integers. If this nonzero entry is on the diagonal of a square matrix, then it corresponds to an integer eigenvalue and the rest eigenvalues are zeros. If it is off the diagonal, then the corresponding matrix is nilpotent, which admits all integer eigenvalues. Thus, any square matrix with a single nonzero integer entry always has an integral spectrum.
	
	Suppose then there are two nonzero elements $a, b$ in $\mathcal{S}$. Place these two entries at positions $(1,2)$ and $(2,1)$ of a square matrix $M$, so its eigenvalues are $\pm\sqrt{ab}$ and $n-2$ zeros, thus proving the condition~\eqref{eq:two} is necessary. Conversely, denote by $(i_1, j_1) $ and $(i_2, j_2)$ the occupied positions in any matrix $M$ for two nonzero entries, and we only need to consider the off-diagonal case, i.e., $i_\alpha\neq j_\alpha$ for $\alpha=1, 2$. Indeed, these occupied positions form a directed graph. If they form a directed $2$-cycle, then the above spectral calculation applies. Otherwise, this graph is acyclic, of which the matrix can be transformed to a triangular matrix whose diagonal entries are integers by a permutation similarity. This completes the proof that every square matrix with two nonzero integers  satisfying~\eqref{eq:two} has integral spectrum and thus $\mathcal{S}$ is $\PSI_n$.
\end{proof}

Generally, a useful observation can be obtained as follows.

\begin{lemma}\label{lem:scale} Suppose every element of an integer multiset $\mathcal{S}$ is divisible by an integer $k>0$. Then, $\mathcal{S}$ is $\PSI_n$ if and only if $\mathcal{S}/k$ is $\PSI_n$.
\end{lemma}
\begin{proof}
	A matrix realisation of the integer multiset $\mathcal{S}$ is denoted by $M$, and correspondingly, there is $M/k$ for $\mathcal{S}/k$. If $\mathcal{S}$ is $\PSI_n$, then each eigenvalue of $M$ is an integer, and obviously, each eigenvalue of $M/k$ is rational. As the characteristics polynomial of $M/k$ is monic in $\Z$ with rational roots, $M/k$ has all integer eigenvalues. Applying the same argument to every pair of $M$ and $M/k$ yields that $\mathcal{S}/k$ is also $\PSI_n$. Then, it is straightforward to verify that if $M/k$ has integer eigenvalues, all eigenvalues of $M$ are integers too, implying that $\mathcal{S}$ is $\PSI_n$.
\end{proof}

It follows immediately that the $\PSI$ property remains unchanged under factor reduction. In particular, denote by $e_2(x)$ for the exponent of $2$ in a nonzero integer $x$, and set $e_2(0)=+\infty$. Background
on the elementary valuation properties can be found in~\cite{Gouvea2020}.
\begin{definition}
	A nonzero integer multiset is dyadically primitive if at least one of its elements is odd.
\end{definition}

Every nonzero multiset becomes dyadically primitive after division by its largest common power of $2$, without changing the $\PSI$ property. Moreover, entrywise reduction of an integer matrix $M$ modulo $2$ gives a binary matrix $\overline{M}$ containing only zeros and ones, with the characteristic polynomial
\begin{equation}\label{reduction}
	\chi_{\overline{M}}(\lambda)=\overline{\chi_M(\lambda)}.
\end{equation}
Here, the bar on the polynomial means coefficientwise reduction. Further, if $\chi_M$ splits completely over $\Z$, then $\chi_{\overline M}$ splits completely over $\F_2$.

\section{$\PSI$ in dimension three}\label{sec:three}

Before presenting our main results, we first establish some useful arithmetic ingredients about the dyadically primitive multiset in dimension three. Denote by $o(\mathcal{S})$ the number of odd integers in a multiset $\mathcal{S}$, and evident, $o(\mathcal{S})\leq \nu(\mathcal{S})$. 

\begin{lemma}\label{lem:binary}
	A dyadically primitive $\PSI_3$ multiset $\mathcal{S}$ satisfies $o(\mathcal{S})\in\{1,2,9\}$. Equivalently, there is no dyadically primitive $\PSI_3$ multiset with $3\leq o(\mathcal{S})\leq 8$.
\end{lemma}
\begin{proof}
	Suppose, to the contrary, that a dyadically primitive $\PSI_3$ multiset $\mathcal{S}$ satisfies $3\leq o(\mathcal{S})\le8$. After entrywise reduction module $2$, every $3\times3$ matrix in $\mathcal{M}_3(\mathcal{S})$ becomes a binary with $o(\mathcal{S})$ ones, and correspondingly, every such binary matrix must have a characteristic polynomial that splits over $\F_2$.

	For $o(\mathcal{S})=3,\ldots,8$, consider the following binary matrix after reduction of suitable matrix realisations of $\mathcal{S}$ and compute their characteristic polynomials
	 \[ \begin{array}{c|c|c} o(\mathcal{S})&M&\chi_{M}(\lambda)\\ \hline 
	 	3&\left(\begin{smallmatrix}1&1&0\\1&0&0\\0&0&0\end{smallmatrix}\right) &\lambda(\lambda^2+\lambda+1)\\ 
	 	4&\left(\begin{smallmatrix}1&1&1\\1&0&0\\0&0&0\end{smallmatrix}\right) &\lambda(\lambda^2+\lambda+1)\\ 
	 	5&\left(\begin{smallmatrix}1&1&1\\1&0&1\\0&0&0\end{smallmatrix}\right) &\lambda(\lambda^2+\lambda+1)\\ 
	 	6&\left(\begin{smallmatrix}1&1&1\\1&1&0\\1&0&0\end{smallmatrix}\right) &\lambda^3+\lambda+1\\ 
	 	7&\left(\begin{smallmatrix}1&1&1\\1&1&0\\0&1&1\end{smallmatrix}\right) &\lambda^3+\lambda^2+1\\ 
	 	8&\left(\begin{smallmatrix}1&1&1\\1&1&1\\1&0&1\end{smallmatrix}\right) &\lambda(\lambda^2+\lambda+1). \end{array} \] 
 	
 It is found that the characteristic polynomials contain the irreducible factor $\lambda^2+\lambda+1$ for $o(\mathcal{S})=3,4,5,8$, and the cubic polynomials have no root in $\F_2$ and are irreducible for $o(\mathcal{S})=6,7$. None of these polynomials splits completely over $\F_2$, thus contradicting the assumption that $\mathcal{S}$ is $\PSI_3$.
\end{proof}

Particularly, any $3\times 3$ integer matrix can be written as
\begin{equation}\label{3matrix}
	M=\begin{pmatrix}a&b&c\\d&e&f\\g&h&i\end{pmatrix},
\end{equation}
and its characteristic polynomial is given by
\begin{equation}
	\chi_M(\lambda)=\lambda^3-T\lambda^2+C\lambda-D,
\end{equation}
where the characteristic coefficients are
\begin{align}
    	T&=a+e+i, \label{trace}\\
		C&=ae+ai+ei-bd-cg-fh, \label{minor}\\
		D&=aei+bfg+cdh-ceg-bdi-afh. \label{determinant}
\end{align}

\begin{lemma}\label{lem:cubic}
 	Assume that a matrix $M$ has the characteristic polynomial in $\Z$, i.e., $\chi_M(\lambda)\in\Z(\lambda)$, and its determinant $D\ne0$ with $d=e_2(D)>0$. If
	\begin{equation}\label{eq:cubic-test}
		e_2(T)>\frac d3,\qquad e_2(C)>\frac{2d}{3},
	\end{equation}
	then its characteristic polynomial does not split completely over $\Q$.
\end{lemma}
\begin{proof}Suppose that the characteristic polynomial of $M$ has three rational roots, so they should be nonzero integers, due to the rational root theorem and the nonzero determinant. Define $\rho_1\le\rho_2\le\rho_3$ as the $2$-adic valuations of these three integer eigenvalues of $M$. Following from Vieta's relations yields $\rho_1+\rho_2+\rho_3=d=e_2(D)>0$.
	
	If $\rho_1<\rho_2$, then the eigenvalue of valuation $\rho_1$ is the unique term of least valuation in the matrix trace~(\ref{trace}). Thus, $\rho_1=e_2(T)>d/3$, forcing all three eigenvalue valuations to be at least $d/3$ and hence contradicting that the sum of three valuations is $d$. If $\rho_1=\rho_2<\rho_3$, then the product of the first two eigenvalues is the unique term of least valuation in $C$ as per~(\ref{minor}), and thus $e_2(C)=\rho_1+\rho_2>2d/3$, again leading to the same contradiction. For the remaining possibility of $\rho_1=\rho_2=\rho_3=t$, dividing each eigenvalue by $2^t$ yields three odd integers, and their sum is odd too. Thus, $e_2(T)=t=d/3$, contrary to~\eqref{eq:cubic-test}. This completes the proof.
\end{proof}

This lemma enables to derive the following result.

\begin{lemma}\label{lem:two-odd}
	A dyadically primitive $\PSI$ multiset $\mathcal{S}$ with $\nu(\mathcal{S})\ge3$ cannot have exactly two odd elements. 
\end{lemma}
\begin{proof}
	Suppose, to the contrary, that a dyadically primitive $\PSI$ multiset $\mathcal{S}$ has two odd elements $u, v$, with $\nu(\mathcal{S})\ge3$.  Consider its matrix realisation as
	\begin{equation}\label{eq:two-odd-placement}
		M=\begin{pmatrix}a&u&c\\d&e&v\\x&h&i\end{pmatrix},
	\end{equation}
where $x$ is a nonzero even entry of least positive $2$-adic valuation $\alpha=e_2(x)$, and except for $u, v$, all other entries in $M$ have valuation at least $\alpha$, with zeros allowed. 
	
	Substituting matrix entries directly into the characteristic coefficients in~(\ref{trace})-(\ref{determinant}) yields
	\begin{equation}
		T=a+e+i,\quad C=ae+ai+ei-ud-cx-vh,
	\end{equation}
	and
	\begin{equation}\label{2determinant}
		D=aei+uvx+cdh-cex-udi-avh.
	\end{equation}
	In the determinant~(\ref{2determinant}), the term $uvx$ has valuation $\alpha$ and every other term is zero or has valuation at least $2\alpha$. Consequently,
	\begin{equation}\label{eq:two-odd-bounds}
		e_2(T)\ge\alpha,\qquad e_2(C)\ge\alpha,\qquad e_2(D)=\alpha.
	\end{equation}
	These inequalities obey the conditions in Lemma~\ref{lem:cubic}, by choosing $d=\alpha>0$, so the characteristic polynomial of $M$ cannot split completely over $\Q$ and hence $M$ does not have integral spectrum. This contradicts that $\mathcal{S}$ is $\PSI$.
\end{proof}

Combining these lemmas from~\ref{lem:binary} to~\ref{lem:two-odd} yields the first main result:

\begin{theorem}\label{thm:three}
	Let $\mathcal{S}$ be a multiset of nine integers containing at least one zero. Then $\mathcal{S}$ is $\PSI_3$ if and only if it has at most one nonzero element, or
	\begin{equation}\label{eq:three-class}
		\mathcal{S}=\{a, b, 0^{[7]}\},\qquad a,b\ne0,\qquad ab=q^2\quad, q\in\Z.
	\end{equation}
	In particular, no $\PSI_3$ multiset has $3\le\nu(\mathcal{S})\le8$.
\end{theorem}

\begin{proof} If a multiset $\mathcal{S}$ of nine integers has at least one zero, then its density $\nu(\mathcal{S})\leq 8$. First, it follows from Proposition~\ref{prop:two} that $\nu(\mathcal{S})\le2$ can be settled and  $3\le\nu(\mathcal{S})\le8$ is left. Next, implementing dyadical reduction on the latter case and using Lemma~\ref{lem:scale} leads to the dyadically primitive multiset $\mathcal{S}$ with $o(\mathcal{S})\leq \nu(\mathcal{S})\leq 8$. Then, Lemma~\ref{lem:binary} excludes $3\leq o(\mathcal{S})\leq 8$, and further Lemma~\ref{lem:two-odd} excludes $o(\mathcal{S})=2$.
	
Finally, the remaining $o(\mathcal{S})=1$ is excluded as follows. Let $u$ be the unique odd element. Since $\nu(\mathcal{S})\ge3$, the multiset contains at least two nonzero even elements. Choose two nonzero even elements $x$ and $y$ whose valuations are the two smallest, and label them so that $\alpha=e_2(x)\le\beta=e_2(y)$. Arrange $u,x,y$, together with a zero entry, as the following matrix realisation of $\mathcal{S}$
	\begin{equation}\label{eq:one-odd-placement}
		M=\begin{pmatrix}a&u&c\\0&e&x\\y&h&i\end{pmatrix}.
	\end{equation}
Every remaining entry then has valuation at least $\beta$. Correspondingly,
	\begin{equation}
		T=a+e+i,\quad C=ae+ai+ei-cy-xh,\quad D=aei+uxy-cey-axh.
	\end{equation}
There is $e_2(T)\ge\beta$, and every term of $C$ has valuation at least $\alpha+\beta$. In the determinant, the term $uxy$ has valuation $\alpha+\beta$ and the others have valuations at least $3\beta,3\beta,$ and $\alpha+2\beta$, respectively. Thus, we have
	\begin{equation}\label{eq:one-odd-bounds}
		e_2(T)\ge\beta,\qquad e_2(C)\ge\alpha+\beta,\qquad e_2(D)=\alpha+\beta.
	\end{equation}
	Set $d=\alpha+\beta$. Since $0<\alpha\le\beta$, one has $\beta>d/3$ and $d>2d/3$. Lemma~\ref{lem:cubic} excludes the complete splitting characteristic polynomial of $M$, further excluding that $\mathcal{S}$ is $\PSI_3$.
\end{proof}

For the remaining $\nu(\mathcal S)=9$, there exists a trivial $\PSI_3$ multiset which contains nine equal elements $a$, so every matrix has three integer eigenvalues $3a,0,0$. It is conjectured that this is the only fully nonzero case.

\begin{conjecture}\label{conj:dense-rigidity} If a multiset of nine nonzero integers is $\PSI$ in dimension three, then all nine elements are equal. \end{conjecture}

Performing dyadical reduction on the multiset with $\nu(\mathcal{S})=9$ leads to

\begin{corollary}\label{cor:dense-dyadic-reduction} 
	After dyadic reduction, every fully nonzero $\PSI_3$ multiset satisfies either $o(\mathcal S)=1$ or $o(\mathcal S)=9. $
\end{corollary}

However, dyadic reduction does not resolve these two cases, by noting that for $o(\mathcal{S})=1$, the matrix realisation used in the proof of Theorem~\ref{thm:three} needs not have a unique determinant term of least valuation when no zero is available, and for $o(\mathcal{S})=9$, reduction modulo $2$ gives rise to the all-ones matrix whose characteristic polynomial splits completely over $\F_2$.

\section{Block embedding in higher dimensions}\label{sec:embedding}

Note that if a $3\times 3$ integer matrix with a noninteger eigenvalue occurs as a diagonal block of a block-triangular matrix, then the high-dimensional block-triangular matrix does not have integral spectrum. Thus, embedding a $3\times 3$ block with noninteger eigenvalues gives rise to the second main result.

\begin{theorem}\label{thm:embedding}
  For $n\ge3$, $\mathcal{S}$ is a multiset of $n^2$ integers, satisfying
	\begin{equation}\label{eq:sparse-bound}
		\nu(\mathcal{S})\le n^2-3n+8.
	\end{equation}
	Then, $\mathcal{S}$ is $\PSI_n$ if and only if it has at most one nonzero entry, or it has the form~\eqref{eq:two}. In particular, no $\PSI_n$ multiset has $3\le\nu(\mathcal{S})\le n^2-3n+8$.
\end{theorem}
\begin{proof} It follows first from Proposition~\ref{prop:two} that we only need consider $\nu(\mathcal{S})\ge3$. Define $k=\min\{\nu(\mathcal{S}),8\}$. Choose $k$ nonzero entries and $9-k$ zeros to configure a $3\times 3$ matrix $M_1$, in a general form of~(\ref{eq:one-odd-placement}), such that it has the nonintegral spectrum as shown in Theorem~\ref{thm:three}. 

Using further $3(n-3)$ zeros to embed $M_1$ into an $n\times n$ square matrix
	\begin{equation}\label{eq:block}
		M=\begin{pmatrix}M_1& \boldsymbol{0}\\ *&M_2\end{pmatrix}.
	\end{equation}
	This needs $(9-k)+3(n-3)=3n-k$ zeros, while it follows from~(\ref{eq:sparse-bound}) that $\mathcal S$ contains at least $n^2-\nu(\mathcal S)\geq 3n-8$ zeros. If $\nu(\mathcal S)\le8$, then $k=\nu(\mathcal S)$ and $3n-\nu(\mathcal S)\le n^2-\nu(\mathcal S)$ for $n\geq 3$. If $\nu(\mathcal S)\ge8$, then $k=8$ and the requirement becomes $3n-8\leq n^2-\nu(\mathcal S)$.  All other elements are arranged in the blocks marked $*$ and $M_2$. It indicates that it is always possible to construct such an $M$ from $\mathcal{S}$.
	
	However, the characteristic polynomial of $M$ admits the decomposition of
	\begin{equation}
		\chi_{M}(\lambda)=\chi_{M_1}(\lambda)\chi_{M_2}(\lambda).
	\end{equation}
    If $\chi_M$ splits completely over $\Z$, then so would $\chi_{M_1}$. This contradicts the choice of $M_1$, so $\mathcal S$ cannot be $\PSI$ in dimension $n$. 
\end{proof}

The threshold $3n-8$ in the proof is the number of zeros used by this particular embedding: one inside a $3\times 3$ matrix and $3(n-3)$ beside it. Whether this threshold is optimal is unknown. Furthermore, there is

\begin{corollary}\label{cor:dense-range}
	If a $\PSI$ multiset of $n^2$ integers with $n\geq 3$ has at least three nonzero entries, then it has at most $3n-9$ zero entries.
\end{corollary}

\section{Nontrivial $\PSI$ multisets in higher dimensions}\label{sec:defects}

It follow above that the multiset density plays a critical role in determining  whether an integer multiset is $\PSI$ or not, and especially, it is generically not $\PSI$ for the sparse case where the portion of nonzero elements in a multiset is less than $1-3/n$. In this section, a second useful parameter is introduced as the number of elements that differ from a common one labelled as $r$ to construct nontrivial $\PSI$ in higher dimensions. 

Denote by $\boldsymbol{e}_i$ the $i$-th orthonormal basis vector and $\mathbf1=\sum_i\boldsymbol{e}_i$ the all-ones column vector. Let $J_n=\one\one^{\mathsf T}$ and $E_{ij}=\boldsymbol{e}_i\boldsymbol{e}_j^{\mathsf T}$ for $i, j=1,\cdots, n$, with the transpose operation $\mathsf T$. Specifically, consider a family of $n\times n$ matrices in the form of
\begin{equation}\label{eq:UV}
	M=rJ_n+\sum_{\alpha=1}^k\delta_\alpha E_{i_\alpha j_\alpha}=UV^{\mathsf T},
\end{equation}
where
\begin{equation*}
	U=[\one,\boldsymbol{e}_{i_1},\ldots,\boldsymbol{e}_{i_k}],~~~V=[r\one,\delta_1\boldsymbol{e}_{j_1},\ldots,\delta_k\boldsymbol{e}_{j_k}].
\end{equation*}
Here $\delta_\alpha$ denotes the nonzero deviation between the matrix entry at position $(i_\alpha,j_\alpha)$ and $r$. Thus, $M$ has rank at most $k+1$. 

Defining further $\mathbf 1_{\{P\}}=1$ when $P$ holds and $0$ otherwise, we are able to obtain the $(k+1)\times (k+1)$ matrix
 \begin{equation}\label{VU}
 R\equiv V^{\mathsf T}U=\begin{pmatrix}nr&r&r&\cdots&r\\ \delta_1&\delta_1\mathbf 1_{\{j_1=i_1\}}&\delta_1\mathbf 1_{\{j_1=i_2\}}&\cdots&\delta_1\mathbf 1_{\{j_1=i_k\}}\\ \delta_2&\delta_2\mathbf 1_{\{j_2=i_1\}}&\delta_2\mathbf 1_{\{j_2=i_2\}}&\cdots&\delta_2\mathbf 1_{\{j_2=i_k\}}\\ \vdots&\vdots&\vdots&\ddots&\vdots\\ \delta_k&\delta_k\mathbf 1_{\{j_k=i_1\}}&\delta_k\mathbf 1_{\{j_k=i_2\}}&\cdots&\delta_k\mathbf 1_{\{j_k=i_k\}}\end{pmatrix},
 	 \end{equation}
which has entries
\begin{equation}\label{eq:R}
	R_{00}=nr,\quad R_{0\beta}=r,\quad R_{\alpha0}=\delta_\alpha, \quad R_{\alpha\beta}=\delta_\alpha\one_{\{j_\alpha=i_\beta\}}	\quad \alpha,\beta=1,\cdots, k.
\end{equation}
Then, we have

\begin{proposition}\label{prop:reduction}
	If $k<n$, then
	\begin{equation}\label{eq:reduction}
		\chi_M(\lambda)=\lambda^{n-k-1}\det(\lambda I_{k+1}-V^{\mathsf T}U).
	\end{equation}
	For a fixed $k$, the number of positional types is bounded independently of $n$.
\end{proposition}
\begin{proof}
	Following first from the Sylvester's identity~\cite{HornJohnson2013} and~(\ref{VU}) gives rise to the characteristic polynomial~(\ref{eq:reduction}). Then, associate a directed edge $i_\alpha\to j_\alpha$ with each deviation $\delta_\alpha$ in position $(i_\alpha,j_\alpha )$, also allowing for loops, i.e., $i_\alpha=j_\alpha$. There are at most $2k$ incident vertices, and simultaneous row and column permutations relabel these vertices, so there are only finite graph types for a fixed $k$. For each type, $n$ enters~\eqref{eq:R} only through $R_{00}=nr$. Thus, the number of types is bounded independently of $n$.
\end{proof}

Next, we consider the low-rank case of $k\leq 2$ and show that there exist nontrivial $\PSI$ multisets in infinitely many dimensions. Note that $k=0$ recovers the constant $\PSI_n$ multiset where all elements are equal.

\subsection{$k=1$}

If $r=0$ in $M$ as per~(\ref{eq:UV}), then $k=1$ means one single nonzero entry, which has been covered by Proposition~\ref{prop:two}. For $r\ne0$, there are two positional types of a single deviation that it lies on or off the diagonal of a square matrix.

\begin{theorem}\label{thm:one-defect} For $r\neq 0$ and $n\geq 2$, the integer multiset $\{r^{[n^2-1]},s\}$ is $\PSI$ if and only if both
	\begin{equation}\label{eq:one-defect}
		(t-(n-1))^2+4(n-1),\qquad 4t+n^2-4
	\end{equation}
	are squares in $\Q$, with $t=s/r \in \Q$.
\end{theorem}
\begin{proof}
 By using Proposition~\ref{prop:reduction}, we obtain that one deviation on the diagonal of matrix $M$ leads to the characteristic polynomial
 \begin{equation}
 	\chi_{\rm d}(\lambda)=\lambda^{n-2}[\lambda^2-(s+(n-1)r)\lambda+(n-1)(s-r)r],\label{eq:one-diag}
 \end{equation}
and an off-diagonal deviation gives
\begin{equation}
	\chi_{\rm o}(\lambda)=\lambda^{n-2}[\lambda^2-nr\lambda+(r-s)r].\label{eq:one-off}
\end{equation}

Computing the discriminant of the quadratic term in the above polynomial and then dividing it by $r^2$ yields 
\begin{equation*}
	\Delta_{\rm d}/r^2=(t-(n-1))^2+4(n-1)
\end{equation*}
for the diagonal deviation, and
\begin{equation*}
	\Delta_{\rm o}/r^2=4t+n^2-4
\end{equation*}
for the off-diagonal deviation. Both quadratic terms split completely over $\Q$ if and only if these two quantities are squares in $\Q$. Furthermore, as the characteristic polynomials are monic with integer coefficients, every rational root is an integer and thus $M$ has integral spectrum. This proves the claim.
\end{proof}

\begin{corollary}\label{cor:examples}
	Let $q\ge5$ be an odd integer and $n=(q^2-5)/4$. For every nonzero integer $r$, the multiset
	\begin{equation}\label{eq:examples}
		\{r^{[n^2-1]},(n+2)r\}
	\end{equation}
	is $\PSI$ in dimension $n$.
\end{corollary}
\begin{proof}
	Choose $s=(n+2)r$ in \eqref{eq:one-defect} and thus $t=n+2$. Then, $(t-(n-1))^2+4(n-1)=4n+5=q^2$ and $4t+n^2-4=(n+2)^2$, both being squares in $\Z$, so it follows from Theorem~\ref{thm:one-defect} that the multiset~(\ref{eq:examples}) is $\PSI$. Moreover, the two nonzero eigenvalues are $r(2n+1\pm q)/2$ for a diagonal deviation and $r(n+1), -r$ for an off-diagonal deviation.

\end{proof}

The smallest example in the above multiset is $\{1^{[24]},7\}$ for $n=5$, of which its nonzero eigenvalues are $8,3$ or $6,-1$, depending on the position of $7$. This example further implies that the unique
multiset proposed in Conjecture~\ref{conj:dense-rigidity} does not extend to higher dimensions.

\subsection{$k=2$}

We first consider that the two nonzero deviations are equal, i.e., $\delta_1=\delta_2=\delta$, and then explore the nonequal case.
\begin{theorem}\label{thm:equal}
For integer $r,s$, and $n\geq 3$, the multiset
	\begin{equation}\label{eq:equal}
		\{r^{[n^2-2]},s^{[2]}\}\text{ is }\PSI_n	\quad\Longleftrightarrow\quad r=0\ \text{or}\ s=r.
	\end{equation}
\end{theorem}
\begin{proof}
	The given multiset has at most two distinct elements, and thus gives rise to the matrix form~(\ref{eq:UV}) with two equal deviations $\delta_1=\delta_2=\delta=s-r$. Place these two deviations at positions $(1,2)$ and $(2,3)$, so $V^{\mathsf{T}}U$ becomes 
		\begin{equation*}
		\begin{pmatrix}nr&r&r\\\delta&0&\delta\\\delta&0&0\end{pmatrix},
	\end{equation*}
	and the characteristic polynomial is
	\begin{equation}\label{eq:equal-path}
		\chi_M(\lambda)=\lambda^{n-3}[\lambda^3-nr\lambda^2-2r\delta \lambda-r\delta^2].
	\end{equation}

	The discriminant of the cubic term in the above polynomial is
	\begin{equation}\label{eq:equal-disc}
		-r^2\delta^2
		\bigl[27\delta^2+(36n-32)r\delta+4n^2(n-1)r^2\bigr].
	\end{equation}
The term in the bracket is quadratic in $\delta$ and has discriminant $-16(3n-4)^3r^2$. Consequently, it is strictly positive, and \eqref{eq:equal-disc} is negative when $r\delta\ne0$. It follows further that the cubic term has one real root and a nonreal conjugate pair of roots, and hence the multiset is $\PSI_n$ only if $r= 0$ or $\delta=0$.
	
 Finally, $r=0$	gives at most two nonzero entries with square product $s^2$, while $\delta=0$, or equivalently, $s=r$, gives a constant matrix. Both satisfy the permutation condition by Proposition~\ref{prop:two} and thus lead to the conclusion that the given multiset is $\PSI_n$.
\end{proof}

It shows that the equal case only allows for constant $\PSI_n$ multisets, and it is also remarked that the proof uses no arithmetic hypothesis on the nonzero deviation and complements the dyadic argument used in Section~\ref{sec:three}.

We then consider the integer multiset $\{r^{[n^2-2]}, s, t\}$ where $n\ge3$, $s=r+\delta_1$, and $t=r+\delta_2$, with $\delta_1\neq \delta_2$. Again, place two distinct deviations in the matrix at positions $(1,2)$ and $(2,3)$, and the characteristic polynomial is $\lambda^{n-3}\bigl[\lambda^3-nr\lambda^2-r(\delta_1+\delta_2)\lambda-r\delta_1\delta_2\bigr].$ Similarly, other positions of the two deviations give factors of degree at most three in the polynomial. Therefore, we only need to study the cubic term in the polynomial.

For simplicity, we now focus on the dimension three. Given three distinct integers $r,s,t$, with $r\ne0$,  define 
\begin{equation}\label{eq:normalized-defects}
	x=(s-r)/r,\qquad y=(t-r)/r,\qquad z=\lambda/r.
\end{equation}
Thus, $x,y\in\Q\setminus\{0\}$ and $x\ne y$, and  we obtain

\begin{proposition}\label{prop:eight}
	The multiset $\{r^{[7]},s,t\}$ is $\PSI$ in dimension three if and only if the following eight polynomials split completely over $\Q$:
	\begin{equation}\label{eq:eight}
		\begin{aligned}
			p_{1x}&=z[z^2-(3+x)z+2x-y],\\
			p_{1y}&=z[z^2-(3+y)z+2y-x],\\
			p_2&=z[z^2-3z-x-y],\\
			p_3&=z^3-(3+x+y)z^2+(xy+2x+2y)z-xy,\\
			p_{4x}&=z^3-(3+x)z^2+(2x-y)z+xy,\\
			p_{4y}&=z^3-(3+y)z^2+(2y-x)z+xy,\\
			p_5&=z^3-3z^2-(xy+x+y)z+xy,\\
			p_6&=z^3-3z^2-(x+y)z-xy.
		\end{aligned}
	\end{equation}
	The $72$ ordered placements of the two deviations form $12$ orbits under permutation similarity and give exactly these eight polynomial types.
\end{proposition}
\begin{proof}
	For the multiset $\{r^{[7]},s,t\}$, dividing  its matrix realisation~(\ref{eq:UV}) by $r$ yields a normalised matrix  $\overline{M}=J_3+xE_{i_1j_1}+yE_{i_2j_2}$. Representatives of the position pair $(i_1,j_1)$ and $(i_2, j_2)$ are listed in the table below, where the first position carries the normalised deviation $x$ and the second carries $y$.
	
	\begin{center} \renewcommand{\arraystretch}{1.12} 
		\begin{tabular}{c p{0.62\textwidth}} 
			\hline 
			Polynomial type & Ordered pairs of deviation positions \\ 
			\hline $p_{1x}$ & $((1,1),(1,2))$; $((1,1),(2,1))$ \\ 
			$p_{1y}$ & $((1,2),(1,1))$; $((1,2),(2,2))$ \\ 
			$p_2$ & $((1,2),(1,3))$; $((1,2),(3,2))$ \\ 
			$p_3$ & $((1,1),(2,2))$ \\ 
			$p_{4x}$ & $((1,1),(2,3))$ \\ 
			$p_{4y}$ & $((1,2),(3,3))$ \\ 
			$p_5$ & $((1,2),(2,1))$ \\ 
			$p_6$ & $((1,2),(2,3))$; $((1,2),(3,1))$ \\ 
			\hline
	 \end{tabular} 
 \end{center}

   These representatives are inequivalent under vertex relabelling, each of which has a trivial stabilizer under six vertex permutations. As a consequence, each of the $12$ positions has size six and thus cover all $9\times 8=72$ ordered pairs of deviation positions. Then, computing the characteristic polynomial $\det(zI_3-\overline{M})$ for each order pair yields the eight polynomials in~(\ref{eq:eight}).
   
    The $\PSI$ property requires that all eight polynomials in~(\ref{eq:eight}) completely split over $\Q$. Conversely, if all eight split completely over $\Q$,  then every normalized eigenvalue $z$ is rational, and thus $\lambda=rz$ is a rational root of a monic integer characteristic polynomial and is an integer.
\end{proof}

\section{Conclusion and remaining problems}\label{sec:summary}

We have studied whether an integer matrix and all its entry permutations can simultaneously have integral spectrum. By using dyadic reduction and Vieta relations, we obtain in Theorem~\ref{thm:three} that the multiset classification with at one zero entry in dimension three is complete, and then extend it to Theorem~\ref{thm:embedding} for higher dimensions, with a linear threshold on the number of zeros. We also introduce a low-rank reduction to construct infinitely many nonconstant $\PSI$ multisets.  

Many problems remain open. For example, it is left open in Conjecture~\ref{conj:dense-rigidity} whether every $\PSI$ multiset with fully nonzero entries in dimension three is always constant. It is unknown in Proposition~\ref{prop:eight} which rational pairs $(x,y)$ make all eight reduced polynomials for two distinct defects split over $\Q$. More generally, whether the  $p$-adic reduction beyond $p=2$ provides more insight into these problems.



\section*{Acknowledgements} The authors thank Michael J. W. Hall for introducing this interesting problem and for the Pythagorean constructions that motivated the present investigation.

\section*{Declaration of generative AI and AI-assisted technologies in the manuscript preparation process} During the preparation of this work, the authors used OpenAI ChatGPT-6 Astra to assist with language editing, literature checking, and the review of mathematical arguments. After using this service, the authors reviewed and edited the content as needed and take full responsibility for the content of the published article.

\bibliographystyle{elsarticle-num} 
\bibliography{references}

\end{document}